\documentclass[10pt,twoside,final]{amsart}

\usepackage[english]{babel}
\usepackage{graphicx,epstopdf,epsfig}
\usepackage{amsfonts,epsfig,fancyhdr,graphics,amsmath,amssymb}

\title{Conjugacy classes with covering number two in hyperbolic orthogonal groups }

\newtheorem{theorem}{Theorem}[section]

\newtheorem{lemma}[theorem]{Lemma}
\newtheorem{corollary}[theorem]{Corollary}
\newtheorem{remark}[theorem]{Remark}
\newtheorem{example}[theorem]{Example}

\newcommand {\ch }{\operatorname{char}}
\newcommand {\SpG }{\operatorname{Sp}}
\newcommand {\PSp }{\operatorname{PSp}}
\newcommand {\GL }{\mathrm{GL}}
\newcommand {\SL }{\mathrm{SL}}
\newcommand {\PSL }{\mathrm{PSL}}
\newcommand {\PSU }{\mathrm{PSU}}

\newcommand {\rank }{\operatorname{rank}}

\newcommand {\GF }{\mathrm{GF}}

\newcommand {\Idm }{\mathrm{I}}
\newcommand {\Dickson }{\mathrm{D}}

\newcommand {\PC }{\mathrm{PC}}
\newcommand {\SC }{\mathrm{SC}}
\newcommand {\CC }{\mathrm{CC}}

\newcommand {\OG }{\mathrm{O}}
\newcommand {\pOG }{\mathrm{O}^+}
\newcommand {\SOG }{\mathrm{SO}}
\newcommand {\pSOG }{\mathrm{SO}^+}
\newcommand {\pOmega }{\Omega^+}

\newcommand {\M }{\mathrm{M}}

\usepackage{ifdraft}
\usepackage{listlbls}
\ifdraft{\usepackage[outer]{showlabels}
	\usepackage{todonotes}}{}
	\ifdraft{\usepackage[outer]{showlabels}
		\usepackage{todonotes} \usepackage{datetime}}{}

\usepackage[pagebackref,colorlinks,citecolor=red,urlcolor=blue,linkcolor=green, bookmarks=false,hypertexnames=true]{hyperref}
\usepackage{orcidlink}
\usepackage{fancyhdr}
\usepackage{verbatim}

\begin{document}

\bibliographystyle{plain}

\setcounter{page}{1}

\thispagestyle{empty}


\author{Klaus Nielsen}\,\orcidlink{0009-0002-7676-2944}
\email{klaus@nielsen-kiel.de}

\ifdraft{\today \ \currenttime}{}
\pagestyle{fancy}
\fancyhf{}
\fancyhead[OC]{Klaus Nielsen}
\fancyhead[EC]{ Products of conjugacy classes in orthogonal groups}
\fancyhead[OR]{\thepage}
\fancyhead[EL]{\thepage}

\maketitle

\begin{abstract}
There is a conjecture, usually attributed to J. G. Thompson, that a finite simple nonabelian group $G$ has a conjugacy class $\Psi$ of covering number 2; i.e $G = \Psi^2$.
We show that the commutator subgroup $\Omega(V,Q)$ of the orthogonal group $\OG(V,Q)$ of a hyperbolic quadratic vector  space $(V, Q)$ over a field $K$
has a conjugacy class $\Psi$ such that $\Omega(V,Q) = \Psi^2 \cup -\Psi^2$.
\end{abstract}


\section{Introduction} \label{intro-sec}

 Thompson's conjecture has been verified for several groups, e.g. for the alternating  and the sporadic groups. Ellers and  Gordeev \cite{EllersGordeev-1998} proved the Thompson conjecture for all simple groups of Lie type over a field with more than 8 elements.
 For references, see \cite{Lev-1999} or \cite{EllersGordeev-1998}.  
Furthermore, it is known that the following simple groups have a conjugacy class of covering number 2:
\begin{enumerate}
    \item $\PSL(n, K)$  (Lev \cite[Theorem 5]{Lev-1994} for $|K| \ge 4$,
     and  \cite[Theorem 3]{Lev-1999}) for all fields $K$, 
    \item $\PSU(V, K, h)$ if $|K| > 9$ and $h$ is hyperbolic (Bünger \cite[Satz 4.1.19]{Bunger-1997}).
    \end{enumerate}
 
Recently, Larsen and Tiep \cite{LarsenTiep-2025} showed that Thompson's conjecture holds for almost all finite simple groups. So there seems to be still some interest in Thompson's conjecture.

In \cite{KNielsen-2025b} and \cite{KNielsen-2025c}, we have shown that the symplectic groups $\PSp(2n, K)$ and the hyperbolic unitary groups $\PSU(V, K, h)$ ($\dim V \ge 4$) have  a conjugacy class of covering number 2.

In this note, we consider the orthogonal group $\OG(V,Q)$ of a hyperbolic quadratic vector  space $(V, Q)$ over a field $K$. 

We call a transformation   $\varphi  \in \OG(V, Q)$ (and its conjugacy class)   hyperbolic if $\varphi$ if there exists a decomposition $V = U \oplus W$, where $U$ and $W$ are $\varphi$-invariant and totally isotropic ($Q(U) = Q(W) = 0$).

We say that  $\varphi$ (and its conjugacy class)  is strictly hyperbolic if $\varphi|_U$  and $\varphi_W$ are coprimary. The minimal polynomial of a strictly hyperbolic transformation is of the form $\mu = q(x) q^*(x)$, where  $q(x)$ and its reciprocal $q^*(x):= x^{\partial q}q(x^{-1})$ are coprime; cf.  Huppert \cite[1.7 Satz]{Huppert-1980a}. We show

\begin{theorem} \label{theorem-1}
	Let $\dim V \ge 4$.  Let $\Psi_1, \Psi_2$ be  conjugacy classes of $\Omega(V, Q)$. Assume  that $\Psi_1 =  \Idm_2 \perp  \Phi_1$, $\Psi_2 = \Idm_2 \perp \Phi_2$, where  $\Phi_1$ and $\Phi_2$ both are strictly hyperbolic and cyclic.
	Then $\Psi_1 \Psi_2$  contains all nonscalar elements of
	$\Omega(V, Q)$.
\end{theorem}

\begin{corollary} \label{cor-1}
	Let $\dim V \ge 4$. Let $\dim V \ge 6$ if $K = \GF(5)$. If $|K| \le 3$  let $\dim V \ge 8$. Then $\Omega(V, Q) = \Psi^2 \cup -\Idm_{2n}$ for some  conjugacy class  $\Psi$ of $\Omega(V, Q)$.
	 \end{corollary}

\section{Preliminaries.}

In this section, $Q$ may have aritrary Witt index.

\subsection{Dickson invariant.} 
The Dickson invariant  $\Dickson(\varphi)$ of the orthogonal
transformation $\varphi \in \OG(V, Q)$ is the rank of $\varphi -1 \mod 2$. The orthogonal transformations with  Dickson invariant 0 form a subgroup of $\OG(V, Q)$, the special orthogonal group $\SOG(V, Q)$; cf. \cite[11.43 Theorem]{Taylor}.

\subsection{Wall's Spinor norm.}

	Let $\varphi \in \OG(V, Q)$.   Then 
	 $\omega_{\varphi}(u -u\varphi, w -w\varphi) := f_Q(u, w - w\varphi)$
	 defines a nondegenerate bilinear form on the path (or residual space) $V(\varphi -1)$ of $\varphi$, the Wall form of $\varphi$. 
	 The spinor norm $\Theta_W(\varphi)$ of
	 $\varphi$ is defined as the discriminant of $\omega_{\varphi}$.
	 Let $K^* = K-\{0\}$.
	 
	 \begin{lemma}                             \label{SPINOR1}
	 	$\Theta_W(\varphi): \OG(V, Q) \rightarrow K^*/(K^*)^2$ is a homomorphism.
	 \end{lemma} 
	 
	 \begin{proof}
	 See \cite[11.50 Theorem]{Taylor}.
	 \end{proof}
	 		
	\begin{lemma}                             \label{SPINOR2}
		Assume that $Q$ has Witt index $\ne 0$. 
		Then  $\Omega(V,Q) = \ker \Theta \cap \SOG(V, Q)$ except when
		$K = \GF(2), \dim V = 4$, and $Q$ is hyperbolic.
	\end{lemma} 
	
\begin{proof}
	See \cite[11.51 Theorem]{Taylor}.
\end{proof}

 If $\varphi \in \SOG(V, Q)$, then $\Theta_W(\varphi)$ coincides with the classical spinor norm defined by Siegel, Eichler, Dieudonn\'e, and Kneser.

\subsection{The standard hyperbolic (split) orthogonal group.}

For a matrix $M \in \M(n,K)$ let $M'$ denote the transpose of $M$. A matrix $A \in \M(n,K)$ is 
alternating if it is congruent to $A_0 \oplus 0_{n-2m}$, where
\[
A_0 = \left (\begin{array} {cc} 0 & \Idm_m\\ -\Idm_m & 0 \end{array} \right ) \in \GL(n,K). 
\].
Now let $Q$ be hyperbolic. Then
the orthogonal group $\OG(V, Q)$ is isomorphic to the subgroup
 $\pOG(2n,K)$  of  all matrices 
\[
P = \left (\begin{array} {cc} A & B\\ C & D \end{array} \right ) \in \GL(n,K), 
\]
where  $A, B, C, D \in \M(n,K), AB'$ and $CD'$ are alternating, and $AD' + BC' =\Idm_n$.

If $\ch K \ne  2$, then $\pOG(2n,K) = \{P \in \GL(n,K); PGP' = G \}$ is the set of all automorphs $P$ of 
\[
G = \left (\begin{array} {cc} 0 & \Idm_n \\ \Idm_n & 0 \end{array} \right ).
\]
If $\ch K = 2$, then $\pOG(2n,K)$ is a subgroup of the symplectic group $\SpG(2n,K) = \{P \in \GL(n,K); PGP' = G \}$.

\begin{example} 
	Let $A \in \M(n,K)$ be idempotent. If $A = A'$ then
	\[
	S = \left (\begin{array} {cc} A & \Idm_n -A\\  \Idm_n-A & A \end{array} \right ) \in \pOG(2n,K),
	\]
	 and $S$ is an involution with $\Dickson(S) = \rank (A - \Idm_n)$,
\end{example}

\begin{lemma} \label{SPINOR3}
	Let
	\[
	S = \left (\begin{array} {cc} \Idm_n & 0\\  C & \Idm_n \end{array} \right ) \in \pOG(2n,K),
	T = \left (\begin{array} {cc} U & 0\\  0 & U^+ \end{array} \right ) \in \pOG(2n,K),
	\]
	where $U$ is a big transvection ($U$ is unipotent of index 2).
	 Then $\Theta(S), \Theta(T) = 1$.
\end{lemma}

\begin{proof}
	If $\ch K \ne  2$, then $S$  and  $T$ are squares. 
	Let $\ch K = 2$. To compute $\Theta(S)$, we may assume that $C = 0_{n-2m} \oplus  \Idm_m \otimes H$, where $H \in \GL(2,K)$ is alternating.
	Then the Wall form of $S$ is congruent to the alternating matrix $\Idm_m \otimes H$.

We compute $\Theta(T)$ and  may assume that  $U = \Idm_{n-2k} \oplus  \Idm_k \otimes H$.
Again, the Wall form of $T$ is congruent to the alternating matrix $\Idm_k \otimes H$.
\end{proof}

\section{Proof of theorem \ref{theorem-1}}

Let 
\[
P = \left (\begin{array} {cc} A & B \\ C & D \end{array} \right ), 
\]
where $A, B,C,D \in \M(n,K)$. The matrix $A$ is the principal corner
$\PC(P) = \PC_n(P)$ of $P$. The matrix $B$ is the congruence
corner $\CC(P)$. If $\PC(P)$ is regular let $\SC(P) :=D -CA^{-1}B$ denote the Schur complement
of $\PC(P)$ in $P$. 

Let $\Gamma_{\mathrm{L}} = \Gamma_{\mathrm{L}}(2n,K)$ be the set of all $P \in \pOG(2n,K)$ with $\CC(P)=0_n$.Let $\Gamma_{\mathrm{U}}(2n,K) = \{P \in \Gamma(n,K); P' \in \Gamma_{\mathrm{L}}(2n,K)\}$.

In the case of a nonscalar conjugacy class $\Psi$ of a hyperbolic unitary vector space, Bünger \cite{Bunger-1997} has shown that $\Psi$ contains a matrix $P$   such that $\PC(P)$ is nonscalar and $\det \PC(P) = \alpha$ if  $\det \Psi =  \alpha \overline{\alpha}^{-1}$. We want to prove a similar result.

\begin{lemma}                             \label{OGHYP1}
	Let $P \in \pOG(2n,K)$, and
	let  $A = \PC(P)$ be nonsingular. Then $\SC(P) = A^+ :=(A')^{-1}$.
	Further, $\Dickson(P) = 0$ and $\Theta(P) = \det A$. 
\end{lemma}

\begin{proof}
	Let 
	\[
	P = \left (\begin{array} {cc} A & B \\ C & D \end{array} \right ),
	L = \left (\begin{array} {cc} \Idm_n & 0 \\ CA^{-1} & \Idm_n \end{array} \right ),
	Q = \left (\begin{array} {cc} A & 0 \\ 0 & A^+ \end{array} \right ),
	U = \left (\begin{array} {cc} \Idm_n & A^{-1}B \\ 0 & \Idm_n \end{array} \right ).
	\]
	Then $\SC(P) = (D - C^{-1}B) = A^+$,  $L, Q,  U \in \pSOG(2n,K)$, and $P = L Q U$.
	By \ref{SPINOR3}, $\Theta(P_1), \Theta(P_2) = 1$.
	We show that  $\Theta(Q) = \det A$:
	  
	First let $\det A = 1$. Then $A$ is a product of transvections.  It follows from  \ref{SPINOR3}, that $\Theta(Q) = 1$.
	So let $\det A = \alpha \ne 1$. Then $A$ is a product of transvections and a dilatation, and $Q$ is a product of matrix with trivial spinor norm  and a matrix
	\[
	Z = \left (\begin{array} {cc} \Idm _{n-1} \oplus \alpha \Idm_1& 0 \\ 0 &  \Idm _{n-1} \oplus \alpha^{-1} \Idm_1 \end{array} \right ).
	\]
	The Wall form of $Z$ is congruent to 
	\[
	\left (\begin{array} {cc} 0 & 1-\alpha^{-1} \\ 1-\alpha &  0\end{array} \right ).
	\]
	Hence $\Theta(Q) = \Theta(Z) = \alpha (1-\alpha)^2$.
\end{proof}

Let $E_k = \Idm_k \oplus 0_{n-k} \in \M(n,K)$ and 
\[
S_k = \left (\begin{array} {cc} \Idm_n - E_k & E_k\\ E_k & \Idm_n -E_k \end{array} \right ). 
\]

\begin{lemma}                           \label{OGHYP2}
If  $\Gamma_{\mathrm{L}} S_k \Gamma_{\mathrm{L}} \cap \Gamma_{\mathrm{L}} S_m, \Gamma_{\mathrm{L}} \ne \emptyset$, then $m = k$.
\end{lemma}

\begin{proof}
	Assume that   $\Gamma_{\mathrm{L}} S_k \Gamma_{\mathrm{L}} \cap \Gamma_{\mathrm{L}} S_m, \Gamma_{\mathrm{L}} \ne \emptyset$. Then there exist matrices 
	\[
	L = \left (\begin{array} {cc} L_1  & 0\\ X &  L_2 \end{array} \right ),
	\widetilde{L} = \left (\begin{array} {cc} L_3  & 0\\ Y &  L_4 \end{array} \right ) 
	\in  \Gamma_{\mathrm{L}}
	\]
	such that $L S_k = S_m \widetilde{L}$. Hence $L_1 E_k = E_m L_4$, and $m = k$.
\end{proof}

For a matrix $M \in \M(m,K)$ let  $\nu(M) := m - \rank M$ denote the nullity of $M$.

\begin{lemma}                           \label{OGHYP3}
	Let $P \in \pOG(2n,K)$. Then 
	$P \in \Gamma_{\mathrm{L}} S_k \Gamma_{\mathrm{L}}$, where $k = \rank(\CC (P)) \equiv  \nu(\PC (P)) \mod 2$.
\end{lemma}

\begin{proof}
	Let $P_1 = QPR$, where  
	\[
	P=\left (\begin{array} {cc} A & B \\ C & D \end{array} \right ),
	Q= \left (\begin{array} {cc} X & 0 \\ 0 & X^+ \end{array} \right ),
	R=\left (\begin{array} {cc} Y & 0 \\ 0 & Y^+ \end{array} \right ).
	\]
	Then $\CC(P_1) = XBY^+$. 
	Hence we may assume that $B = \Idm_k \oplus 0_{n-k}$.
	Let 
	\[
	A = \left (\begin{array} {cc} A_1 & A_2 \\ A_3 & A_4 \end{array} \right ),
	D = \left (\begin{array} {cc} D_1 & D_2 \\ D_3 & D_4 \end{array} \right ),
	\]
	where $A_1, D_1 \in \M(k,K), A_4,D_4 \in \M(n-k,K)$. Then
	$A_3 = 0$ and $A_1$ is alternating since $AB'$ is alternating
	so that $A_4$ must be nonsingular and we can even achieve that
	$A_4 = \Idm_{n-k}$. And since $B'$ is alternating and $AD' +  BC' = \Idm_{2n}$, it follows that $D_2 = 0$ and $D_4 = \Idm_{n-k}$.
	Put 
	\[
	S = \left (\begin{array} {cc} -A_1 & -A_2 \\ -A_2' & 0 \end{array} \right ),
	T = \left (\begin{array} {cc} -D_1 & -D_3' \\ -D_3 & T_4  \end{array} \right ).
	\]
	Then $S$ and $T$ are alternating and 
	\[
	\widetilde{A} := A+BS =\left (\begin{array} {cc} 0 & 0 \\ 0 & \Idm_{n-k} \end{array} \right ),
	\widetilde{D} := D+TB=\left (\begin{array} {cc} 0 & 0 \\ 0 &  \Idm_{n-k}\end{array} \right ).
	\]
	Now 
	\[
	\left (\begin{array} {cc} I & 0 \\ T & I \end{array} \right )
	\left (\begin{array} {cc} A & B \\ C & D \end{array} \right ) 
	\left (\begin{array} {cc} I & 0 \\ S & I \end{array} \right )
	= \left (\begin{array} {cc} A+BS & B \\ C+DS +T(A+BS) & D +TB\end{array} \right ).
	\]
	Let 
	\[
	\widetilde{C} = C+DS +T(A+BS) =
	\left (\begin{array} {cc} G_1 & G_2 \\ G_3 & G_4 \end{array} \right ).
	\]
	Then $G_2 = 0$ since $\widetilde{C} \widetilde{D}'$ is alternating.
	And $G_3 = 0$ since $\widetilde{A}' \widetilde{C}$ is alternating.
	From $\widetilde{A} \widetilde{D}' + B \widetilde{C}'
	= \Idm_{2n}$ we see that $G_1 =  \Idm_k$. Finally
	$G_4 = 0$ for a suitable matrix $T_4$.
\end{proof}

\begin{lemma}                                            \label{OGHYP4}
	Let $A, B \in \M(n,K)$ such that $\rank (A,B) = n$ and $A B'$ is alternating.
	Furthermore, assume that $\nu(A) \equiv 0 \mod 2$.
	Then there exists an alternating matrix $X \in \M(n,K)$ such that $A+BX$ is nonsingular.
\end{lemma}

\begin{proof}
	Let $R,S \in \GL(n,K)$, and put $\widetilde{A} = RAS, \widetilde{B} = R B S^+$. 
	Then $\widetilde{A} \widetilde{B}' = R AB' R'$ is alternating and $\rank (\widetilde{A},\widetilde{B}) = n$. If $\widetilde{A} + \widetilde{B} \widetilde{X}$ is nonsingular for some alternating matrix $\widetilde{X} \in \M(n,K)$, then
	$ A + B S^+ \widetilde{X} S^{-1} = R^{-1}(\widetilde{A} + \widetilde{B} \widetilde{X} )S^{-1}$ is nonsingular.
	
	Hence we may assume that $A = \Idm_{n-\nu} \oplus 0_{\nu}$.
	Let 
	\[
	A = \left (\begin{array} {cc} \Idm_{n-\nu} & 0 \\ 0 & 0_v \end{array} \right ),
	B = \left (\begin{array} {cc} B_1 & B_2 \\ B_3 & B_4 \end{array} \right ),
	X = \left (\begin{array} {cc} 0_{n-\nu} & 0 \\ 0 & Y_{\nu} \end{array} \right ),
	\]
	where $B$ is partitioned as $A$ and $Y_{\nu} \in \GL(\nu,K)$ is alternating.
	Then $B_3 = 0$ and $B_4$ must be regular. Then 
	\[
	A+BX = \left (\begin{array} {cc} \Idm_{n-\nu} & B_2 Y_{\nu} \\ 0 & B_4 Y_{\nu}\end{array} \right )
	\]
	is nonsingular.
\end{proof}

\begin{lemma}                                          \label{OGHYP5}
	Let $n \ge 2$. Let $P \in \pSOG(2n,K)$. Then $P$ is conjugate to a matrix $\widetilde{P}$ such that $\PC(\widetilde{P})$ is nonsingular. 
\end{lemma}

\begin{proof}
	Let 
	\[
	P = \left (\begin{array} {cc} A & B \\ C & D \end{array} \right ).
	\]
		By lemma \ref{OGHYP3}, $\nu(A) \equiv 0 \mod 2$.
	Hence by lemma \ref{OGHYP4}, $A + B X$ is nonsingular for some alternating matrix $X$.  
	Then $P$ is conjugate to
	\[
	\widetilde{P} = \left (\begin{array} {cc}  \Idm_n & 0 \\ -X & \Idm_n \end{array} \right )
	\left (\begin{array} {cc} A & B \\ C & D \end{array} \right )
	\left (\begin{array} {cc}  \Idm_n & 0 \\ X & \Idm_n \end{array} \right )
	= \left (\begin{array} {cc} A + B X & \ast \\ \ast & \ast\end{array} \right )
	\]
\end{proof}

\begin{lemma}                                \label{OGHYP6}
	Let $n \ge 3$. Let $\delta \in K^*$. Let $H \in \M(n,K)$ be alternating and $H \ne 0$.
	There exist  alternating matrices $X, Z \in \M(n,K)$ such that 
	\begin{enumerate}
		\item  $\det (\Idm_n + HZ) = \det (\Idm_n + HZ) = \delta^2$ and
		\item $\rank  H(X -  Z) \in \{1, 2\}$.
	\end{enumerate}
\end{lemma}

\begin{proof} 
	
	First let $\rank H = 2$. Then we may assume that $n = 3$ and
	\[
	H  = \left (\begin{array} {ccc} 0 & 1 & 0 \\ -1 & 0 & 0  \\  0 & 0 & 0  \end{array} \right ).
	\]
	
	\[
	X  = \left (\begin{array} {ccc} 0 & 1 -\delta & 0 \\ \delta -1 & 0 & 0  \\  0 & 0 & 0  \end{array} \right ),
	Z  = \left (\begin{array} {ccc} 0 & 1 -\delta & 1 \\ \delta -1 & 0 & 0  \\  -1 & 0 & 0  \end{array} \right ).
	\]
		Then
	\[
	HX  = \left (\begin{array} {ccc} \delta -1  & 0  & 0 \\ 0 & \delta -1 & 0  \\  0 & 0 & 0  \end{array} \right ),
	HZ  = \left (\begin{array} {ccc} \delta -1  & 0  & 0 \\ 0 & \delta -1 & -1  \\  0 & 0 & 0  \end{array} \right ).
	\]

	So let $\rank H \ge 4$. Then we may assume that $n = 4$ and	
	\[
	H  = \left (\begin{array} {cccc} 0 & 1 & 0 & 0\\ -1 & 0 & 0  & 0 \\  0 & 0 & 0  & 1  \\ 0 & 0 & -1  & 0  \end{array} \right ).
	\]
		Let
	\[
	X  = \left (\begin{array} {cccc} 0 & 1-\delta & 0 & 0\\ \delta -1 & 0 & 0  & 0 \\                            0 & 0 & 0  & 0  \\ 0 & 0 & 0  & 0  \end{array} \right ),
	Z  = \left (\begin{array} {cccc} 0 & 1-\delta & 1 & 0\\ \delta -1 & 0 & 0  & 0 \\                            -1 & 0 & 0  & 0  \\ 0 & 0 & 0  & 0  \end{array} \right ).
	\]
	Then
	\[
	HX  = \left (\begin{array} {cccc} \delta - 1& 0 & 0 & 0\\ 0 & \delta -1 & 0  & 0 \\                            0 & 0 & 0  & 0  \\ 0 & 0 & 0  & 0  \end{array} \right ),
	HZ   = \left (\begin{array} {cccc} \delta - 1 & 0 & 0 & 0\\ 0 & \delta -1 & 0  & -1 \\                         0  & 0 & 0  & 0  \\ 1 & 0 & 0  & 0  \end{array} \right ).
	\]
\end{proof}

\begin{lemma}                                  \label{OGHYP7}
	Let $n \ge 2$. Let $A \in \GL(n,K)$.  
	 If $AZ = ZA^+$ for all alternating  $Z \in \M(n,K)$, then $A = \pm \Idm_n$ or $n  = 2$
	 and $\det A = 1$.
\end{lemma}

\begin{proof}
	Suppose that $A$ is not scalar. 
	Let $n  = 2$. If
	\[
	 \left (\begin{array} {cc} 0 & 1\\ \gamma & \delta  \end{array} \right )
	 \left (\begin{array} {cc} 0 & 1\\ -1 & 0  \end{array} \right )
	 = \left (\begin{array} {cc} 0 & 1\\ -1 & 0  \end{array} \right )
	 \left (\begin{array} {cc} -\gamma^{-1} \delta & 1\\ \gamma^{-1} & 0  \end{array} \right ),
	\]
	then $\gamma =-1$.
	So let  $n \ge 3$. We may additionally assume that $\PC_2(A)$ is singular. Put $Z = H \oplus 0_{n-2}$, where
	$H$ is regular and  alternating. Then  $\PC_2(AZA') = \PC_2(A) H \PC_2(A’)$ is singular, hence $AZA' \ne Z$.
\end{proof}

\begin{lemma}                               \label{OGHYP8}
	Let $P \in \pOG(2n,K)$ be non-scalar. Assume further that $\CC (P') = 0$. Then $P$ is $\Gamma_{\mathrm{U}}$-conjugate to $\widetilde{P}$
	with $\CC(\widetilde{P}) \ne 0$ except when $n = 2$ and $\PC(P) \in \SL(2,K)$.
\end{lemma}

\begin{proof}
	We may assume that $\CC (P) = 0$. Hence
	\[
	P =   \left (\begin{array} {cc} A & 0\\ 0 &  A^+ \end{array} \right ).
	\]
	By the preceeding lemma there exists an alternating matrix $S \in \M(n,K)$   such that $AS -SA \ne 0$. Then
	\[
	\left (\begin{array} {cc} \Idm_n & -S\\ 0 & \Idm_n \end{array} \right )
	\left (\begin{array} {cc} A & 0\\ 0 & A' \end{array} \right )
	\left (\begin{array} {cc} \Idm_n & S\\ 0 & \Idm_n \end{array} \right )
	= \left (\begin{array} {cc} A & AS -SA'\\ 0 & A^+ \end{array} \right )
	\]
	has a nontrivial congruence corner.
\end{proof} 

\begin{lemma}                                          \label{OGHYP9}
	Let $n \ge 3$. Let $\delta \in K^*$.
	Let $P \in \pSOG(2n,K)$ be non-scalar.  
	Then $P$ is conjugate to a matrix $\widetilde{P}$ such that $\PC(\widetilde{P})$ is nonscalar and
	$\det \PC(\widetilde{P}) = \delta^2 \det \PC(P)$.
\end{lemma}

\begin{proof}
	By lemma \ref{OGHYP5}, we may assume that $\PC(P)$ is nonsingular. Let 
	\[
	P = \left (\begin{array} {cc} A & B \\ C & D \end{array} \right ).
	\]
	By lemma \ref{OGHYP8}, we may assume that $B \ne 0$ (If $C \ne 0$ consider $P^G$).
	By lemma \ref{OGHYP6}, there exist  alternating matrices $X, Z \in \M(n,K)$ such that 
	\begin{enumerate}
		\item  $\det (\Idm_n + A^{-1}BZ) = \det (\Idm_n + A^{-1}BZ) = \delta^2$ and
		\item $\rank  A^{-1}B(X -  Z) \in \{1, 2\}$.
	\end{enumerate}
Clearly $A + BX$ and $A + BZ$ cannot be both scalar as $BX -BZ$ is nonscalar.

 So let $A + BX$ be nonscalar. Now
	\[
	\left (\begin{array} {cc} \Idm_n & 0 \\ -X & \Idm_n \end{array} \right )
	 \left (\begin{array} {cc} A & B \\ C & D \end{array} \right )
	\left (\begin{array} {cc} \Idm_n & 0 \\ X & \Idm_n \end{array} \right )
	= \left (\begin{array} {cc} A + B X & \ast \\ \ast & \ast\end{array} \right ).
	\]
\end{proof}

\begin{proof}[Proof of theorem \ref{theorem-1}]
We assume that $\Psi_1, \Psi_2 \subseteq \pOmega(2n,K)$.
	By Witt's cancellation theorem, the fix space of  an element of $\Psi_1$ or 
	$\Psi_2$ is a hyperbolic plane. Hence   $\Psi_1$ and $\Psi_2$ are even conjugacy classes of $\pOG(2n,K)$. Let $P \in \pOmega(2n,K)$ be nonscalar.
		Let $\mu (\Psi_1) = (x-1)q_1 q_1^*, \mu (\Psi_2) = (x-1)q_2 q_2^*$, where $\gcd (q_1, q_1^*) = \gcd (q_2, q_2^*) =1$. 
	We have shown above that $P$ is $\Gamma_{\mathrm{L}}$ conjugate to a matrix 
	\[
	Q =   \left (\begin{array} {cc} A & B\\ C &  D \end{array} \right ),
	\]
	where $A$ is nonscalar, $\det A = q_1(0)q_2(0)$.

	By \cite[Theorem 1.2]{BN-1999}, $A = R S$ is the product of cyclic matrices $R$ and $S$ with
	$\mu (R) = (x-1)q_1, \mu (R) = (x-1)q_2$. By \ref{OGHYP1}, $\SC(Q) = R^+S^+$. Then $Q = Q_1 Q_2$, where
	\[
	Q_1 = \left (\begin{array} {cc} R & 0\\ CS^{-1} &  R^+ \end{array} \right ),
	Q_2 = \left (\begin{array} {cc} S & R^{-1}B\\ 0 &  S^+ \end{array} \right ).
	\]
	By \ref{OGHYP1}, factors have zero Dickson invariant.  Hence by \ref{OGHYP9}, $Q_1 \in \Psi_1,  Q_2 \in \Psi_2$. 
\end{proof}

\section{Proof of corollary \ref{cor-1}}

\begin{lemma}                               \label{OGHYP10}
	Let  $n \ge 2$. If $|K|\le 3$ let $n \ge 4$.
	If $K = \GF(5)$ let $n \ge 3$.
	There exists a hyperbolic conjugacy class 
	$\Psi$ of  $\pOmega(2n,K)$ with $\Psi = \Psi^{-1}$ such that $\Psi^2$
	contains all nonscalar transformations of $\pOmega(2n,K)$.
\end{lemma}

\begin{proof}
	 Let $\Psi$ the $\pOmega(2n,K)$-conjugacy class of 
	\[
	 \left (\begin{array} {cc} W & 0\\ 0 &  W^+ \end{array} \right ),
	\]
	where $W \in \GL(n, K)$ is cyclic with minimal polynomial
	\begin{enumerate}
		\item $(x-1)(x-\lambda^2)^t$ if $|K^2| \ge 4$ and  $\lambda^2 \ne 0, \pm 1$, 
		\item $(x-1)(x^2+ x -1)^t$ if $K = \GF(5)$ and $n = 2t+1 \ge 3$,
		\item $(x-1)(x^2+ x -1)^t (x^3 + x -1)$ if $K = \GF(5)$ and $n = 2t+4 \ge 4$,
		\item $(x-1)(x^3 + x^2 -1)^t$ if $K = \GF(3)$ and $n = 3t+1 \ge 4$,
		\item $(x-1)(x^3 + x^2 -1)^t (x^2+ x -1)^2$ if $K = \GF(3)$ and $n = 3t+5 \ge 5$,
		\item $(x-1)(x^3 + x^2 -1)^t (x^3 -x +1)^2$ if $K = \GF(3)$ and $n = 3t+5 \ge 6$,
		\item $(x-1)(x^3 + x +1)^t$ if $K = \GF(2)$ and $n = 3t+1 \ge 4$,
		\item $(x-1)(x^3 + x +1)^t (x^4 + x^2 +1) $ if $K = \GF(2)$ and $n = 3t+5 \ge 5$,
		\item $(x-1)(x^3 + x +1)^t (x^5 + x^2 +1) $ if $K = \GF(2)$ and $n = 3t+6 \ge 6$.
	\end{enumerate}
	Then $\Psi$ is an $\pOG(2n,K)$-conjugacy class so that $\Psi = \Psi^{-1}$.
	By theorem \ref{theorem-1}, $\Psi^2$ contains all nonscalar elements of 
	$\pOmega(2n,K)$.
\end{proof}

\begin{remark} 
	There exists a unique noncentral class $\Sigma$ of involutions in $\Omega(4,5)$. 
	It is easy to see that $-\Idm_4, \Idm_4, \Sigma \subseteq \Sigma^2$. By \cite[Theorem 8.5]{KT-1998}, $\Omega(4,5)$ is bireflectional. Hence $\Omega(4,5) \subseteq \Sigma^2$. 
\end{remark}

Using the computer algebra GAP \cite{GAP2026}, we can show

\begin{remark}                               \label{OGHYP11}
	Let $q \in \{2, 3, 5\}$.
	There exist exactly 2  conjugacy classes $\Psi, \Phi$ of $\Omega(6,q)$
	with minimal polynomial $(x^3 + x^2 -1) (x^3 -x -1)$. We have $\Phi =  \Psi^{-1}$ and
\begin{enumerate}
\item 	$\Omega(6,q) = \Psi \Psi^{-1} = \Psi^2  \cup \Idm_6$ if $q \in \{2, 3\}$,
\item 	$\Omega(6,q) =  \Psi \Psi^{-1} \cup -\Idm_6 = \Psi^2  \cup \Idm_6 \cup -\Idm_6$
if $q = 5$.
\end{enumerate}
\end{remark}

\begin{remark}                               \label{OGHYP12}
	Let $q \in \{3, 5\}$.
	There exist a unique   conjugacy class $\Psi$ of $\Omega(6,q)$
	with minimal polynomial $(x^2 + x -1) (x^2 -x -1) (x + 1)$. We have 
	\begin{enumerate}
		\item 	$\Omega(6,q)  = \Psi^2$ if $q = 3$,
		\item 	$\Omega(6,q)  = \Psi^2  \cup -\Idm_6$ if $q = 5$.
	\end{enumerate}
\end{remark}


\ifdraft{\listoflabels}{}

\end{document}

\typeout{get arXiv to do 4 passes: Label(s) may have changed. Rerun}